\documentclass[11pt,oneside]{article}
\usepackage{times,amsmath,amsbsy,amssymb,amscd,mathrsfs,graphicx}
\usepackage{amssymb}
\usepackage{amsmath,bm}
\usepackage{amsmath,amsthm}
\usepackage{framed}
\usepackage{xcolor,color}
\usepackage[font=small,labelfont=md,textfont=it]{caption}
\usepackage{floatrow,float}
\usepackage[titletoc, title]{appendix}
\usepackage[colorlinks,linkcolor=blue,citecolor=blue]{hyperref}
\usepackage{longtable}
\usepackage{pgfplots}
\usepackage{diagbox}
\usepackage{booktabs,makecell,multirow}
\usepackage[capitalise, nosort]{cleveref}
\usepackage{algorithm}
\usepackage{algorithmic}
\usepackage{amsmath,bm}
\usepackage{cases}
\usepackage{geometry}
\usepackage{extarrows}
\usepackage{tcolorbox}
\usepackage{bbm}
\usepackage{syntonly}
\usepackage[figuresright]{rotating}
\usepackage{epstopdf}
\usepackage{graphicx}
\usepackage{tikz}
\usepackage{lipsum}
\usepackage{xcolor} % 引入颜色宏包
\usepackage{verbatim}
\usepackage{enumerate}
\tikzset{elegant/.style={smooth,thick,samples=50,black}}
\tikzset{eaxis/.style={->,>=stealth}}

\usepackage[misc]{ifsym}

\crefname{equation}{}{}
\crefname{lem}{Lemma}{Lemmas}
\crefname{thm}{Theorem}{Theorems}
\crefname{assum}{Assumption}{Assumptions}

\newcommand{\proj}[0]{ {\bf proj}}

\newcommand{\conv}[1]{{\bf conv}\left\{ {#1} \right\}}

\newcommand{\R}{\,{\mathbb R}}

\newcommand{\haus}[0]{ {\bf haus}}

\newcommand{\ssnm}[1]
{
	\left\vert\kern-0.25ex
	\left\vert\kern-0.25ex
	\left\vert
	{#1}
	\right\vert\kern-0.25ex
	\right\vert\kern-0.25ex
	\right\vert
}

\makeatletter
\def\spher@harm#1{%
	\vbox{\hbox{%
			\offinterlineskip
			\valign{&\hb@xt@2\p@{\hss$##$\hss}\vskip.2ex\cr#1\crcr}%
		}\vskip-.36ex}%
}
\def\gshone{\spher@harm{.}}
\def\gshtwo{\spher@harm{.&.}}
\def\gshthree{\spher@harm{.&.&.}}
\let\gsh\spher@harm
\makeatother

\newtheorem{coro}{Corollary}[section]

\newtheorem{Def}{Definition}[section]
\newtheorem{assum}{Assumption}

\newtheorem{lem}{Lemma}[section]

\newtheorem{thm}{Theorem}[section]

\newcolumntype{I}{!{\vrule width 1,5pt}}
\newlength\savedwidth

\newlength\savewidth

\newcounter{mnote}
\let\oldmarginpar\marginpar
\renewcommand\marginpar[1]
{\-\oldmarginpar[\raggedleft\footnotesize #1]{\raggedright\footnotesize #1}}

\makeatletter\def\@captype{table}\makeatother

\newfloatcommand{capbtabbox}{table}[][\FBwidth]
\usepackage[T1]{fontenc} 
\usepackage[utf8]{inputenc} 
\usepackage{authblk}
\usepackage[all]{xy}
\begin{document}
	%	\title{
		%		\Large \bf An accelerated gradient flow approach to convex multiobjective programming}
	\title{\Large \bf Multiobjective Balanced Gradient Flow: A Dynamical Perspective on a Class of Optimization Algorithms\thanks{nobody}}
	
	\author[,1]{Yingdong Yin\thanks{Email: yydyyds@sina.com}}
	
%	\author[,1]{Liping Tang\thanks{Email: tanglipings@163.com}}

	%\affil[1]{National Center for Applied Mathematics in Chongqing, Chongqing Normal University, Chongqing, 401331, China}
	%\affil[2]{Chongqing Research Institute of Big Data, Peking University,  Chongqing, 401121, China}

	\date{\today }
	\maketitle
	
	\begin{abstract} 
    This paper proposes a novel dynamical system called the Multiobjective Balanced Gradient Flow (MBGF), offering a dynamical perspective for normalized gradient methods in a class of multi-objective optimization problems. Under certain assumptions, we prove the existence of solutions for MBGF trajectories and establish their convergence to weak Pareto points in the case of convex objective functions. For both convex and non-convex scenarios, we provide convergence rates of $O(1/t)$ and $O(1/\sqrt{t})$, respectively.  
	\end{abstract}

	%	\end{keywords}	
	\section{Introduction}
       In this paper, we define the Euclidean space $\mathbb{R}^n$ and consider the following unconstrained multi-objective optimization problem:
       \begin{equation}\label{eq:MOP} 
       \min_{x \in \mathbb{R}^n} f(x) = \left( f_1(x), f_2(x), \cdots, f_m(x) \right)^\top \tag{MOP}
       \end{equation} 
       where $f_i(x)$ are smooth functions. For problem \cref{eq:MOP}, this paper primarily studies a class of dynamical systems unique to multiobjective optimization (distinct from single-objective optimization), which we call the \textit{Multiobjective Balanced Gradient Flow}:
       \begin{equation}\label{eq:MBGF-Intro}
       \dot{x}(t) + \proj_{C_\alpha(x(t),t)}(0) = 0 \tag{MBGF}
       \end{equation} 
       where $C_\alpha(x(t),t) = \textbf{conv}\left\{ \frac{\nabla f_i(x(t))}{\alpha(x(t),t)} \mid i = 1, 2, \cdots, m \right\}$, representing the convex hull of the normalized gradients with respect to $\eta \geq 0$.
      \subsection{Gradient Methods for Multiobjective Optimization}
      
      To address the inherent limitations of traditional scalarization methods in multiobjective optimization, Fliege et al. \cite{fliege2000steepest} proposed the multiobjective steepest descent method. Its descent direction at each iteration can be obtained by solving a quadratic programming problem:
      \begin{equation}\label{eq:steepestdescent} 
      d(x) = -\mathbf{proj}_{C(x)}(0) 
      \end{equation}
      where $C(x) = \mathbf{conv}\{\nabla f_i(x) \mid i=1,\cdots,m\}$. In recent years, numerous gradient methods for multiobjective optimization have been studied, including Newton methods \cite{fliege2009newton,wang2019extended}, quasi-Newton methods \cite{povalej2014quasi,lapucci2023limited,qu2011quasi,prudente2024global,ansary2015modified,yang2024global2}, trust-region methods \cite{carrizo2016trust,ramirez2022nonmonotone}, Barzilai-Borwein methods \cite{morovati2016barzilai,chen2023barzilai,yang2024global}, conjugate gradient methods \cite{lucambio2018nonlinear,gonccalves2020extension,gonccalves2022study}, conditional gradient methods \cite{assunccao2021conditional,chen2023conditional}, proximal gradient methods \cite{tanabe2019proximal,tanabe2022globally,tanabe2023accelerated,tanabe2023convergence}, and others. 
      
      Since \cref{eq:steepestdescent} represents the minimum-norm vector in $C(x)$, the algorithm may converge slowly for unbalanced problems where there are significant differences in the gradients' magnitudes across objective functions. Major studies addressing such unbalanced problems include \cite{katrutsa2020follow,liu2021conflict,yang2024global} , where the core idea involves gradient preprocessing. Katrutsa et al. \cite{katrutsa2020follow} and Yang et al. \cite{yang2024global2} adopted similar normalization schemes and achieved promising experimental results using distinct step size criteria. The direction selection through this approach can be summarized as:
      \begin{equation} 
      d(x) = -\proj_{C_\eta(x)}(0) 
      \end{equation} 
      where $C_\eta(x) = \mathbf{conv}\left\{ \frac{\nabla f_i(x)}{\|\nabla f_i(x)\| + \eta} \mid i=1,2,\cdots,m \right\}$ with $\eta \geq 0$.
      
      \subsection{Multiobjective Gradient Flows}
      
      Dynamical systems provide a novel theoretical perspective for algorithmic research and have been extensively studied in single-objective optimization. A particularly distinctive class of such systems is the normalized gradient system \cite{CORTES20061993,murray2019revisiting}:
      \begin{equation} \label{eq:single-normalized-flow}
      \dot{x}(t) + \frac{\nabla f(x(t))}{\|\nabla f(x(t))\|} = 0 
      \end{equation}
      Due to the inherent challenges in studying non-smooth dynamical systems, and because its trajectory solutions coincide with those of the conventional gradient descent dynamical system
      \begin{equation} 
      \dot{x}(t) + \nabla f(x(t)) = 0
      \end{equation}
      (as shown in \cite{murray2019revisiting}), this system has not received broader research attention. In multiobjective optimization, Attouch et al. \cite{Attouch2014} first utilized the following dynamical system to study the multi-objective steepest descent method:
      \begin{equation} 
      \dot{x}(t) + \proj_{C(x(t))}(0) = 0
      \end{equation}
      Similarly, studies employing dynamical systems to investigate various gradient-type multiobjective algorithms include \cite{sonntag2024fast,sonntag2024fast2,bot2024inertial,luo2025accelerated,yin2025multiobjective}.
      
      Notably, unbalanced problems are unique to multiobjective optimization (compared to single-objective optimization), yet current research lacks a dynamical systems perspective on these issues. Therefore, we employ \cref{eq:MBGF-Intro} to provide new insights into this class of problems. Although \cref{eq:MBGF} incorporates "normalized" gradients, $\|\dot{x}(t)\|$ is not always equal to $1$ – a key characteristic distinguishing it from \cref{eq:single-normalized-flow}. For these reasons, we avoid designating it as a normalized gradient flow to emphasize its inherent multiobjective optimization properties.
    \section{Prelimilary}
	\subsection{Notation}
	In this paper, \(\mathbb{R}^d\) denotes a \(d\)-dimensional Euclidean space with the inner product \(\langle \cdot, \cdot \rangle\) and the induced norm \(\|\cdot\|\). For vectors \(a, b \in \mathbb{R}^d\), we say \(a \leq b\) if \(a_i \leq b_i\) for all \(i = 1, \ldots, d\). Similarly, the relations \(a < b\), \(a\ge b\) and \(a>b\) can be defined in the same way.	 \(\R_{+}^d:=\{x\in \mathbb R^n \mid x\ge 0\}\). The open ball of radius \(\delta\) centered at \(x\) is denoted by \(B_{\delta}(x) := \{ y \in \mathbb{R}^d \mid \|y - x\| < \delta \}\). The set \(\Delta^d := \{ \theta \in \mathbb{R}^d \mid \theta \geq 0 \text{ and } \sum_{i=1}^d \theta_i = 1 \}\) is the positive unit simplex. For a set of vectors \(\{\eta_1, \ldots, \eta_d\} \subseteq \mathbb{R}^d\), their convex hull is defined as \(\conv{\{\eta_1, \ldots, \eta_d\}} := \{ \sum_{i=1}^d \theta_i \eta_i \mid \theta \in \Delta^m \}\). For a closed convex set \(C \subseteq \mathbb{R}^d\), the projection of a vector \(x\) onto \(C\) is  \(\proj_C(x) := \arg\min_{y \in C} \|y - x\|^2\). 
	\subsection{Pareto optimal}
		\begin{Def}[\cite{miettinen1999nonlinear}]\label{def:defofpareto}
		Consider the multiobjective optimization problem \cref{eq:MOP}.
		\begin{itemize}
			\item[$\rm (i)$] A point \( x^* \in \mathbb{R}^n \) is called a Pareto point or a Pareto optimal solution  if there has no \(y\in \mathbb R^n\) that \(F(y)\leq F(x^*)\) and $F(y)\neq F(x^*)$. The set of all Pareto points is called the Pareto set and is denoted by \( \mathcal{P} \).  The image \(F(\mathcal P)\) of the Pareto set $\mathcal P$ is the Pareto front.
			
			\item[$\rm (ii)$] A point \( x^* \in \mathbb{R}^n \) is called a weak Pareto point or weakly Pareto optimal solution if there has no \(y\in \mathbb R^n\) that \(F(y)<F(x^*)\). The set of all weak Pareto points is called the weak Pareto set and is denoted by \( \mathcal{P}_w \). The image \(F(\mathcal P_w)\) of the Pareto set $\mathcal P_w$ is the weak Pareto front.
		\end{itemize}  
	\end{Def}

\begin{Def}\label{def:KKTcondition}
	A point \( x^* \in \mathbb{R}^n \) is said to satisfy the Karush-Kuhn-Tucker (KKT) conditions if there exists \( \theta \in \Delta^m  \) such that
	\begin{equation}\label{eq:KKTpoint}
	\sum_{i=1}^m \theta_i \nabla f_i(x^*) = 0 
	\end{equation}
	If \( x^* \) satisfies the KKT conditions, it is called a Pareto critical point. The set of all Pareto critical points is called the Pareto critical set and is denoted by \( \mathcal{P}_c \).
\end{Def}	

\begin{lem}
The following statements holds:
\begin{itemize}
	\item[(a)] If $x_*$ is local weakly Pareto optimal for \cref{eq:MOP}, then $x_*$ is a Pareto critical point for \cref{eq:MOP};
	\item[(b)] If $f$ is convex and $x_*$ is Pareto  critical for \cref{eq:MOP}, then $x_*$ is weakly Pareto optimal for \cref{eq:MOP};
	\item[(c)] If $f$ is strictly convex and $x_*$ is Pareto critical for \cref{eq:MOP}, then $x_*$ is Pareto optimal for \cref{eq:MOP}. 
\end{itemize}
\end{lem}
	
	\subsection{Merit function}
	A merit function refers to a nonnegative function in \cref{eq:MOP} that attains zero only at weak Pareto points. In this paper, we also consider the merit function  established by Tanabe et al. \cite{tanabe2024new} 
	\begin{equation}\label{eq:meritfunction} 
	u_0(x) := \sup_{z \in \mathbb{R}^n} \min_{i=1,\ldots,m} f_i(x) - f_i(z) 
	\end{equation}
	This function mirrors the role of \( f(x) - f(x^*) \) in single objective optimization, being nonnegative and indicating weak optimality, as formalized in the following theorem.
	
	\begin{thm}[\cite{tanabe2024new}]\label{thm:weakpareto} 
		Let \( u_0:\mathbb R^n \to \R \) be defined as in \cref{eq:meritfunction}. Then,
		\begin{itemize} 
			\item[\(\rm (i)\)]\( u_0(x) \geq 0 \) for all \( x \in \mathbb{R}^n \);
			\item[\(\rm (ii)\)] \( x \in \mathbb{R}^n \) is a weak Pareto point of \cref{eq:MOP} if and only if \( u_0(x) = 0 \);
			\item[\(\rm (iii)\)] \( u_0(x) \) is lower semicontinuous.
		\end{itemize}
	\end{thm}
	\begin{proof}
		See \cite[Theorem 3.1, 3.2]{tanabe2024new}.
	\end{proof} 
	
	\begin{Def}[\cite{Attouch2014}]\label{def:levelset}
		Let \( F: \mathbb{R}^n \to \mathbb{R}^m \), \( F(x) = (f_1(x), \ldots, f_m(x))^\top \), be a vector-valued function, and let \( a \in \mathbb{R}^m \). The  level set is defined as
		\[
		\mathcal{L}(F, a) := \{ x \in \mathbb{R}^n : F(x) \leq a \} = \bigcap_{i=1}^m \{ x \in \mathbb{R}^n : f_i(x) \leq a_i \}
		\]
		Moreover, we denote
		\[
		\mathcal{LP}_w(F, a) := \mathcal{L}(F, a) \cap \mathcal{P}_w
		\]
	\end{Def}
    \subsection{Proposition of porjection}
     \begin{Def}\label{def:Hausdoff-distance} 
     The Hausdorff distance between two closed convex subsets $C$ and $D$ of $\R^n$ is defined by  
     \begin{equation} 
     \haus(C, D) = \max \{e(C, D); e(D, C)\}
     \end{equation}  
     where $e(C, D) = \sup_{x\in C} d(x, D)$ is the excess of $C$ on $D$, and $e(D, C) = \sup_{x\in D}d(x, C)$ is the excess of $C$ on $D$. Equivalently 
     \begin{equation} 
      \haus(C, D) = \sup_{x\in \R^n}|d(x, C) - d(x, D)|.
     \end{equation} 
     \end{Def} 
     \begin{lem}[\cite{Attouch2014}]\label{lem:projection-hausdorff} 
     Let $C$ and $D$ be two closed convex subsets of $\R^n$. Then, for any $x\in \R^n$ the mapping $C \to  \proj_C(x)$ is Hölder-continuous. More precisely, for any two closed  convex subsets $C$ and $D$ of $\R^n$  
     \begin{equation} 
     \|\proj_C(x) -\proj_D(x)\| \le  \rho (\|x\|)\haus(C, D)
     \end{equation}
       where $\rho (\|x\|) = (\|x\| + d(x, C) + d(x, D))$
    \end{lem}
	\subsection{Assumption}
	\begin{assum}\label{assum:L-smoothing}
	Each $f_i$ is lower bounded and $L$-smooth for any $i=1,2,\cdots,m$. That is,  
	$$
	\|\nabla f_i(x)-\nabla f_i(y)\|\le L\|x-y\|\quad \forall \ x,y\in \R^n
	$$
	\end{assum}
    \begin{assum}\label{assume:BoundedGradiet}
    There exist $M>0$, such that 
    $$
    \max_{x\in \R^n}\|\nabla f_i(x)\|\le M\ \text{for all }x\in \R^n
    $$
    \end{assum}
    \begin{assum}\label{assume:BoundedLevelset} 
    For every $a\in \R$, the level set $\mathcal L(f,a)$ is bounded.  
    \end{assum}
	\section{Balanced gradient flow}
	We consider the following dynamical system, termed the \textit{Multiobjective Balanced Gradient Flow}:
	\begin{equation}\label{eq:MBGF}  
	\left\{\begin{aligned}
	&\dot{x}(t) + \mathbf{proj}_{C_\alpha(x(t),t)}(0) = 0\\
	& C_\alpha(x(t),t) = \mathbf{conv}\left\{ \frac{\nabla f_i(x(t))}{\alpha_i (x(t),t)} \mid i = 1, \cdots, m \right\}
	\end{aligned}\right. \tag{MBGF}
	\end{equation} 
	where $\alpha(u,t)$ is Lipschitz continuous function with lower bound $\alpha_{\min}$ and upper bound $\alpha_{\max}$, i.e.
	\begin{equation}\label{eq:alpha-Lipschitz}
	|\alpha_i(u,t)-\alpha_i (v,s)|\le L_{\alpha_i} \|(u,t)-(v,s)\|,
	\end{equation}
	\begin{equation}\label{eq:alpha-bounded}
	\alpha_{\min} \le \alpha_i(u,t)\le \alpha_{\max} .
	\end{equation}
    Given the Cauchy problem:
	\begin{equation} \label{eq:CP}
	\left\{\begin{aligned}
	\dot{x}(t) &= -\mathbf{proj}_{C_\alpha(x(t),t)}(0), \\
	x(0) &= x_0.
	\end{aligned}\right. \tag{CP}
	\end{equation}

    \subsection{Existence of \cref{eq:CP}}
\begin{lem}\label{lem:convexhull-hausdorff}
	Under \cref{assum:L-smoothing,assume:BoundedGradiet,assume:BoundedLevelset}, define the set-valued mapping $C_\alpha : \mathbb{R}^n\times [t_0,+\infty ) \rightrightarrows \mathbb{R}^n$, $(u,t) \mapsto C_\alpha(u,t) = \conv{\frac{\nabla f_i(u)}{\alpha_i(u,t)}\mid i = 1,2,\cdots,m}$ where $\alpha(\cdot)$ defined by \cref{eq:alpha-Lipschitz} and \cref{eq:alpha-bounded}. Then, there exists a constant $K := \frac{L}{\alpha_{\min}} + \frac{L_{\alpha}M}{\alpha_{\min}^2}$ such that for any $(u,t), (v,s) \in \mathbb{R}^n\times [t_0,+\infty)$,
	\begin{equation}
	\haus(C_{\alpha}(u,t), C_{\alpha}(v,s)) \le K \|(u,t) - (v,s)\|.
	\end{equation}
\end{lem}
\begin{proof}
	Take any $\xi \in C_{\alpha}(v,s)$. Then $\xi = \sum_{i=1}^m \lambda_i \frac{\nabla f_i(v)}{\alpha(v,s)}$, where $\lambda = (\lambda_1, \cdots, \lambda_m) \in \Delta^m$. Then
	\begin{equation} 
	\begin{aligned}
	d(\xi, C_\alpha(u,t)) &\le \left\| \xi - \sum_{i=1}^m \lambda_i \frac{\nabla f_i(u)}{\alpha(u,t)} \right\| \\
	&\le \left\| \sum_{i=1}^m \lambda_i \frac{\nabla f_i(v)}{\alpha(v,s)} - \sum_{i=1}^m \lambda_i \frac{\nabla f_i(u)}{\alpha(u,t)} \right\| \\
	&\le \sum_{i=1}^m \lambda_i \left\| \frac{\nabla f_i(v)}{\alpha(v,s)} - \frac{\nabla f_i(u)}{\alpha(u,t)} \right\| \\
    &\le \sum_{i=1}^m \lambda_i \left\|\frac{\alpha(u,t)\nabla f_i(v)-\alpha(v,s)\nabla f_i(u)}{\alpha (v,s)\alpha(u,t)}\right\| \\
	&\le \frac{L}{\alpha_{\min}} \|u - v\|+\frac{L_\alpha M}{\alpha_{\min}^2}\|(u,t)-(v,s)\|\\
    &\le \left( \frac{L}{\alpha_{\min}} +\frac{L_\alpha M}{\alpha_{\min}^2}\right)\|(u,t)-(v,s)\|.
	\end{aligned}
	\end{equation} 
	Based on this and \cref{def:Hausdoff-distance}, we obtain the result. \qedhere
\end{proof}
    \begin{thm}[\cite{sonntag2024fast2}]\label{thm:exist}
    	Let \(\mathcal{X}\) be a real Hilbert space, and let \(\Omega \subset \mathbb{R} \times \mathcal{X}\) be an open subset containing \((t_0, x_0)\). Let \(G\) be an upper semicontinuous map from \(\Omega\) into the nonempty closed convex subsets of \(\mathcal{X}\). We assume that \((t, x) \mapsto \proj_{G(t,x)}(0)\) is locally compact. Then, there exists \(T > t_0\) and an absolutely continuous function \(x\) defined on \([t_0, T]\), which is a solution to the differential inclusion
    	\begin{equation*}
    	\dot{x}(t) \in G(t, x(t)), \quad x(t_0) = x_0.
    	\end{equation*}
    \end{thm}

\begin{thm}
		Under \cref{assum:L-smoothing,assume:BoundedGradiet,assume:BoundedLevelset}, for any $x_0 \in \R^n$, there exist $T > t_0$ and an absolutely continuous function $x(t)$ defined on $[t_0, T]$ that is a solution to \cref{eq:CP} with respect to $\eta > 0$.
\end{thm}

\begin{proof}
	According to the \cref{lem:convexhull-hausdorff} and \cref{lem:projection-hausdorff}, it is straightforward to verify that the mapping $G: u \mapsto -\proj_{C_\alpha(u,t)}(0)$ is upper semicontinuous and closed convex-valued. Then, by \cref{thm:exist}, we obtain the conclusion. 
\end{proof}

\begin{thm}
	Suppose \cref{assum:L-smoothing,assume:BoundedGradiet,assume:BoundedLevelset} hold. Then for any $x_0 \in \mathbb{R}^n$, there exists an absolutely continuous function $x(t)$ defined on $[t_0, +\infty)$ that solves \cref{eq:CP}.
\end{thm}

\begin{proof}
	Define the family of solution sets:
	\[
	\mathfrak{S} := \left\{ x : [t_0, T) \to \mathbb{R}^n : T \in [t_0, +\infty],\ x \text{ is a solution to } \cref{eq:CP} \text{ on } [t_0, T) \right\}.
	\]
	Define the partial order $\preccurlyeq$ on $\mathfrak{S}$ as:
	\[
	\begin{aligned}
	&x_1(\cdot) \preccurlyeq x_2(\cdot) :\Leftrightarrow \\
	&T_1 \le T_2 \text{ and } x_1(\cdot) = x_2(\cdot) \text{ for all } t \in [t_0, T_1).
	\end{aligned}
	\]
	By Zorn's lemma, there exists a maximal element $x(t)$ defined on $[t_0, T)$. Suppose $T < +\infty$. Define:
	\[
	h(t) := \|x(t) - x(t_0)\|.
	\]
	Then:
	\[
	\begin{aligned}
	\frac{d}{dt}\frac{1}{2}h^2(t) &= \langle \dot{x}(t), x(t) - x(t_0) \rangle \le \|\dot{x}(t)\| h(t) \\
	&= \left\| {\proj}_{C_\alpha(x(t),t)}(0) \right\| h(t) \\
	&= \left\| \sum_{i=1}^m \lambda_i(x(t),t) \frac{\nabla f_i(x(t))}{\alpha(x(t),t)} \right\| h(t) \\
	&\le \sum_{i=1}^m \lambda_i \left\| \frac{\nabla f_i(x(t))}{\alpha(x(t),t)} - \frac{\nabla f_i(0)}{\alpha(0,t)} \right\| h(t) + \sum_{i=1}^m \lambda_i \left\| \frac{\nabla f_i(0)}{\alpha(0,t)} \right\| h(t) \\
	&\le \left( \frac{L}{\alpha_{\min}} + \frac{L_\alpha M}{\alpha_{\min}^2} \right) \|x(t)\| h(t) + \frac{M}{\alpha_{\min}}h(t) \\
    &=\frac{M}{\alpha_{\min}}\left[\left(\frac{L}{M}+\frac{L_\alpha }{\alpha_{\min} }\right)\| x(t)\|h(t)+h(t)\right]\\
	&\le c(1 + \|x(t)\|) h(t)
	\end{aligned}
	\]
	where $c = \frac{M}{\alpha_{\min}}\cdot \max \left\{ 1, \frac{L}{M} + \frac{L_\alpha}{\eta_{\min}} \right\}$. Let $\widetilde{c} := c(1 + \|x(t_0)\|)$. Using the triangle inequality, we obtain:
	\[
	\|\dot{x}(t)\| \le \widetilde{c}(1 + \|x(t) - x(t_0)\|).
	\]
	Combining the above inequalities yields:
	\[
	\frac{1}{2}h^2(t) \le \widetilde{c}(1 + h(t))h(t).
	\]
	By a Gronwall-type argument, for any $\varepsilon > 0$ and almost all $t \in [t_0, T - \varepsilon]$,
	\[
	h(t) \le \widetilde{c} T \exp(\widetilde{c} T).
	\]
	Thus $h$ is uniformly bounded on $[t_0, T)$. This implies the boundedness of both $x(t)$ and $\dot{x}(t)$. Since $x(t)$ is absolutely continuous, for any $t \in [t_0, T)$:
	\[
	x(t) = x(t_0) + \int_{t_0}^t \dot{x}(s)  ds
	\]
	and
	\[
	\|x(t') - x(t'')\| \le \sup_{t \in [t_0, T)} \|\dot{x}(t)\| \cdot |t' - t''|.
	\]
	Therefore, $\lim_{t \to T^-} x(t)$ exists. Set:
	\[
	x(T) = x(t_0) + \int_{t_0}^T \dot{x}(s)  ds.
	\]
	Similar to the proof in \cite{sonntag2024fast2}, this implies the existence of $\hat{x}(\cdot)$ defined on $[t_0, T + \delta)$ for some $\delta > 0$, contradicting the maximality of $x(\cdot)$ in $\mathfrak{S}$.
\end{proof}
	\begin{lem}\label{lem:inequalityofCP} 
	Let $x(t)$ be a trajectory solution of \cref{eq:CP}. Then
	\begin{equation} 
	(\alpha(x(t),t)) \|\dot{x}(t)\|^2 + \frac{d}{dt} f_i(x(t)) \leq 0
	\end{equation}
\end{lem} 

\begin{proof} 
	Note that $\frac{\nabla f_i(x(t))}{\alpha(x(t),t)} \in C_\alpha(x(t),t)$. By the projection theorem,
	\begin{equation} 
	\left\langle \frac{\nabla f_i(x(t))}{\alpha(x(t),t)} + \dot{x}(t), \dot{x}(t) \right\rangle \leq 0.
	\end{equation} 
	The result follows from direct computation and observing that $\frac{d}{dt} f_i(x(t)) = \langle \nabla f_i(x(t)), \dot{x}(t) \rangle$.
\end{proof} 

\begin{lem}
	Let $x(t)$ be a trajectory solution of \cref{eq:CP}. Then for any $s \leq t$,
	\begin{equation} 
	\mathcal{L}(f, f(x(t))) \subseteq \mathcal{L}(f, f(x(s)))
	\end{equation} 
\end{lem}
	\subsection{Convergence Analysis in the Convex Case}
	\begin{lem}[\cite{sonntag2024fast2}, Lemma 4.12]\label{lem:diffequaldiffi}
		Let \(\{h_i\}_{i=1,\cdots,m}\) be a set of continuously differentiable functions, \(h_i: [t_0, +\infty) \to \mathbb{R}\). Define \(h: [t_0, +\infty) \to \mathbb{R}\), \(t \mapsto h(t) := \min_{i=1,\cdots,m} h_i(t)\). Then, the following holds:
		
		\noindent\(\rm (i)\) \(h\) is differentiable almost everywhere on \([t_0, +\infty)\);
		
		\noindent\(\rm (ii)\) \(h\) satisfies almost everywhere on \([t_0, +\infty]\) that there exists \(i \in \{1,\cdots,m\}\) such that
		\[
		h(t) = h_i(t) \qquad \frac{d}{dt} h(t) = \frac{d}{dt} h_i(t)
		\]
	\end{lem}
	
	\begin{thm}\label{thm:convergencerate} 
	Let $f_i$ be smooth convex functions with Lipschitz continuous gradients, and let $x(t)$ be a trajectory solution of \cref{eq:CP} for $\eta \geq 0$. Then
	\begin{equation} 
	u_0(x(t)) \leq \frac{R^2 \alpha_{\max}}{t}
	\end{equation} 
	\end{thm} 
	\begin{proof}
	For any $t \geq 0$, select $z \in \mathcal{L}(f, f(x(t)))$. Define functions on $[0, t]$ as follows:
	\begin{equation} 
	\mathcal{E}_z^i(s) = \frac{s}{\alpha_{\max}} (f_i(x(s)) - f_i(z)) + \frac{1}{2} \|x(s) - z\|^2, \quad \forall s \in [0, t]
	\end{equation}
	\begin{equation} 
	\mathcal{E}_z(s) = \frac{s}{\alpha_{\max}} \min_{i=1,\dots,m} (f_i(x(s)) - f_i(z)) + \frac{1}{2} \|x(s) - z\|^2, \quad \forall s \in [0, t]
	\end{equation} 
	Then
	\begin{equation} 
	\begin{aligned}
	\frac{d}{ds} \mathcal{E}_z^i(s) 
	&= \frac{1}{\alpha_{\max}} (f_i(x(s)) - f_i(z)) + \frac{s}{\alpha_{\max}} \langle \dot{x}(s), \nabla f_i(x(s)) \rangle + \langle x(s) - z, \dot{x}(s) \rangle \\
	&\overset{(*)}{\leq} \frac{1}{\alpha_{\max}} (f_i(x(s)) - f_i(z)) + \langle x(s) - z, \dot{x}(s) \rangle \\
	&= \frac{1}{\alpha_{\max}} (f_i(x(s)) - f_i(z)) + \left\langle x(s) - z, -\sum_{i=1}^m \theta_i \frac{\nabla f_i(x(s))}{\alpha(x(s),s)} \right\rangle \\
	&\leq \frac{1}{\alpha_{\max}} (f_i(x(s)) - f_i(z)) - \sum_{i=1}^m \theta_i \frac{1}{\alpha(x(s),s)} \Big( f_i(x(s)) - f_i(z) \Big) \\
	&\leq \frac{1}{\alpha_{\max}} (f_i(x(s)) - f_i(z)) - \frac{1}{\alpha_{\max}} \min_{i=1,\dots,m} \big( f_i(x(s)) - f_i(z) \big)
	\end{aligned}
	\end{equation} 
	where inequality $(*)$ follows from \cref{lem:inequalityofCP}. By the  \cref{lem:diffequaldiffi}, we obtain
	\begin{equation} 
	\frac{d}{ds} \mathcal{E}_z(s) \leq 0.
	\end{equation} 
	Thus,
	\begin{equation} 
	\frac{t}{\alpha_{\max}} \min_{i=1,\dots,m} (f_i(x(t)) - f_i(z)) \leq \mathcal{E}_z(t) \leq \mathcal{E}_z(0) = \frac{1}{2} \|x_0 - z\|^2 \leq R^2.
	\end{equation} 
	where $R=\sup_{x\in \mathcal L(f,f(x(t)))}\|x\|$. After straightforward computation, we have
	\begin{equation} 
	u_0(x(t)) = \sup_{z \in \mathcal{L}(f, f(x(t)))} \min_{i=1,\dots,m} (f_i(x(t)) - f_i(z)) \leq \frac{R^2 \alpha_{\max}}{t}.
	\end{equation} 
  \end{proof}

	 \begin{lem}[Opial's Lemma]\label{lem:opial}
	 Let $S \subseteq \mathbb{R}^n$ be a nonempty set and let $x: [0, +\infty) \to \mathbb{R}^n$. Suppose $x$ satisfies the following conditions:
     \begin{itemize}
         \item[$\rm (i)$] Every limit point of $x$ belongs to $S$;
         \item[$\rm (ii)$] For every $z \in S$, $\lim_{t \to \infty} \|x(t) - z\|$ exists.
     \end{itemize}
	 \noindent Then $x(t)$ converges to some $x^\infty \in S$ as $t \to \infty$.
	\end{lem}
	\begin{proof}
	See \cite{Attouch2014}
	\end{proof}
	
	\begin{thm}
	Let $x: [0, +\infty) \to \mathbb{R}^n$ be a trajectory solution of \cref{eq:CP} for smooth convex functions $f_i$, $i=1,\dots,m$. Assume that each $f_i$ is bounded below and that the conditions of \cref{thm:convergencerate} hold. Then $x(t)$ converges to a weak Pareto point of \cref{eq:MOP}.
	\end{thm} 
	\begin{proof} 
	Define the set
	\begin{equation} 
	S := \left\{ z \in \mathbb{R}^n \mid f_i(z) \leq f_i^\infty \text{ for all } i=1,\dots,m \right\}
	\end{equation} 
	where $f_i^\infty = \lim_{t \to \infty} f_i(x(t))$. The existence of these limits follows from \cref{lem:inequalityofCP} and the boundedness below of $f_i$. Moreover, $x(t)$ is bounded, so it has a limit point $x^\infty \in \mathbb{R}^n$. Thus, there exists a sequence $(x(t_k))_{k \geq 0}$ such that $\lim_{k \to \infty} x(t_k) = x^\infty$. By the continuity of $f_i$, we have
	$$
	f_i(x^\infty) \leq \liminf_{k \to \infty} f_i(x(t_k)) = \lim_{k \to \infty} f_i(x(t_k)) = f_i^\infty.
	$$
	Hence, $x^\infty \in S$. Furthermore, for any $z \in S$, consider the function $h_z(t) = \frac{1}{2} \|x(t) - z\|^2$. Then
	$$
	\dot{h}_z(t) = \langle x(t) - z, \dot{x}(t) \rangle \leq 0,
	$$
	where the inequality follows from \cref{lem:inequalityofCP}. Therefore, $h_z(t)$ is non-increasing, which implies that $\lim_{t \to \infty} \|x(t) - z\|$ exists. By \cref{lem:opial} (Opial's Lemma), $x(t)$ converges to some $x^\infty \in S$. Moreover, by the previous theorem, we know that $\lim_{t \to \infty} u_0(x(t)) = 0$, and by the lower semicontinuity of $u_0(x)$, we have
	$$
	u_0(x^\infty) \leq \liminf_{t \to \infty} u_0(x(t)) = 0.
	$$
	Then, by \cref{thm:weakpareto} of \cref{eq:MOP}.
	\end{proof} 
\subsection{Convergence analysis in the strongly convex case}
In this section, we assume that the objective functions $f_i$ are all $\mu_i$-strongly convex, i.e.,  
\begin{equation} 
f_i(z) \ge f_i(x) + \langle \nabla f_i(x), z - x \rangle + \frac{\mu_i}{2} \|x - z\|^2.
\end{equation} 
Let $\mu = \min_{i=1,\cdots,m} \mu_i$.
\begin{thm}
Suppose \cref{assum:L-smoothing,assume:BoundedGradiet,assume:BoundedLevelset} hold. Let $x(t)$ be the trajectory solution of \cref{eq:CP}, and assume each $f_i$ is $\mu_i$-strongly convex with $\mu \ge  \frac{1}{\alpha_{\max}}$. Then,  
\begin{equation}
u_0(x(t)) = O(e^{-\frac{1}{\alpha_{\max}}t}).
\end{equation}
\end{thm}
\begin{proof}
    Define the functions  
    \begin{equation} 
    \mathcal{W}_i(t) := f_i(x) - f_i(z) + \frac{1}{2} \|x(s) - z\|^2,
    \end{equation} 
    and  
    \begin{equation} 
    \mathcal{W}(t) := \min_{i=1,\cdots,m} (f_i(x) - f_i(z)) + \frac{1}{2} \|x(s) - z\|^2.
    \end{equation}
    For $t > t_0$ and any $s \in [t_0, t]$, we compute  
    \begin{equation} 
    \begin{aligned}
    \frac{d}{ds} {\mathcal{W}}_i(s) &= \left\langle \nabla f_i(x(s)), \dot{x}(s) \right\rangle + \left\langle x(s) - z, \dot{x}(s) \right\rangle \\
    &\le -\frac{1}{\alpha_{\max}} \min_{i=1,\cdots,m} (f_i(x(s)) - f_i(z)) - \frac{\mu}{2} \|x(s) - z\|^2 \\
    &= -\frac{1}{\alpha_{\max}} \left( \min_{i=1,\cdots,m} (f_i(x(s)) - f_i(z)) + \frac{\mu \alpha_{\max}}{2} \|x(s) - z\|^2 \right) \\
    &= -\frac{1}{\alpha_{\max}} \left( \min_{i=1,\cdots,m} (f_i(x) - f_i(z)) + \frac{1}{2} \|x(s) - z\|^2 \right) + \frac{1 - \mu \alpha_{\max}}{2\alpha_{\max}} \|x(s) - z\|^2 \\
    &\le -\frac{1}{\alpha_{\max}} \left( \min_{i=1,\cdots,m} (f_i(x) - f_i(z)) + \frac{1}{2} \|x(s) - z\|^2 \right) \\
    &= -\frac{1}{\alpha_{\max}} \mathcal{W}(t).
    \end{aligned}
    \end{equation}
    Thus, for almost all $s \in [t_0, t]$, we have  
    $$
    \frac{d}{ds} \mathcal{W}(s) + \frac{1}{\alpha_{\max}} \mathcal{W}(s) \le 0.
    $$
    It follows that $\frac{d}{ds} \left( e^{\frac{1}{\alpha_{\max}} t} \mathcal{W}(s) \right) \le 0$. Integrating both sides over $[t_0, t]$ yields  
    $$
    \begin{aligned}
    e^{\frac{1}{\alpha_{\max}} t} \mathcal{W}(t) &\le e^{\frac{1}{\alpha_{\max}} t_0} \min_{i=1,\cdots,m} (f_i(x_0) - f_i(z)) + e^{\frac{1}{\alpha_{\max}} t_0} \frac{1}{2} \|x_0 - z\|^2 \\
    &\le e^{\frac{1}{\alpha_{\max}} t_0} \min_{i=1,\cdots,m} (f_i(x_0) - \inf f_i) + e^{\frac{1}{\alpha_{\max}} t_0} R^2.
    \end{aligned}
    $$
    By the definition of $\mathcal{W}(t)$, we obtain  
    $$
    e^{\frac{1}{\alpha_{\max}} t} \min_{i=1,\cdots,m} (f_i(x(t)) - f_i(z)) + \frac{e^{\frac{1}{\alpha_{\max}} t}}{2} \|x(t) - z\|^2 \le e^{\frac{1}{\alpha_{\max}} t_0} \min_{i=1,\cdots,m} (f_i(x_0) - \inf f_i) + e^{\frac{1}{\alpha_{\max}} t_0} R^2.
    $$
    Due to the non-negativity of the term $\frac{1}{2} \|x(t) - z\|^2$, taking the supremum over $z$ on both sides yields the conclusion.
\end{proof}
\begin{thm}
    Suppose Assumptions 1, 2, and 3 hold. Let $x(t)$ be the trajectory solution of $\rm(CP)$, and assume each $f_i$ is $\mu$-strongly convex with $\mu > \frac{1}{\alpha_{\max}}$. Suppose $x^*$ satisfies $\lim_{t \to \infty} x(t) = x^*$. Then,  
    $$
     \|x(t) - x^*\|^2 = O(e^{-\frac{1}{\alpha_{\max}} t}).
    $$
    
\end{thm}
\begin{proof}
    Clearly, for any $t$, we have $f_i(x(t)) \ge f_i(x^*)$. Therefore,  
    $$
    \|x(t) - x^*\|^2 \le \frac{2 e^{\frac{1}{\alpha_{\max}} t_0} \min_{i=1,\cdots,m} (f_i(x_0) - \inf f_i) + 2 e^{\frac{1}{\alpha_{\max}} t_0} R^2}{e^{\frac{1}{\alpha_{\max}} t}}.
    $$
    This completes the proof.
\end{proof}

\subsection{Convergence analysis in the non-convex case}

\begin{thm}  	Suppose \cref{assum:L-smoothing,assume:BoundedGradiet,assume:BoundedLevelset} hold, and let $x(t)$ be a trajectory solution of \cref{eq:CP} for $\eta > 0$. Then
$$
 \min_{s\in[0,t]} \|\mathbf{proj}_{C_\eta(x(s))}(0)\| \leq \frac{\sqrt{\min\limits_{i=1,\dots,m} (f_i(x_0) - \inf\limits_{x\in\mathbb{R}^n} f_i(x))}}{\sqrt{\eta} \sqrt{t}}
$$
\end{thm} 

\begin{proof} 
By \cref{lem:inequalityofCP}, we have
$$
\|\dot{x}(s)\|^2 \leq -\frac{1}{\eta}\frac{d}{ds}f_i(x(s)) \quad \text{for all } s\in[0,t].
$$
Integrating both sides over $[0,t]$ and using the definition of the equation yields
$$
t \min_{s\in[0,t]} \|\mathbf{proj}_{C_\eta(x(s))}(0)\|^2 \leq \int_0^t -\frac{d}{ds}f_i(x(s)) ds = \frac{1}{\eta}(f_i(x_0) - \inf_{x\in\mathbb{R}^n} f_i(x)).
$$
The conclusion follows after straightforward computation.
\end{proof}
\begin{thm} 
	Suppose \cref{assum:L-smoothing,assume:BoundedGradiet,assume:BoundedLevelset} hold, and let $x(t)$ be a trajectory solution of $\rm (CP)$ for $\eta = 0$. Assume there exists no $x^*\in\mathbb{R}^n$ such that $\max\limits_{i=1,\dots,m} \|\nabla f_i(x^*)\| = 0$ for all $i=1,\dots,m$. Then
$$
\min_{s\in[0,t]} \|\mathbf{proj}_{C_\eta(x(s))}(0)\| \leq \frac{\sqrt{\min\limits_{i=1,\dots,m} (f_i(x_0) - \inf\limits_{x\in\mathbb{R}^n} f_i(x))}}{\sqrt{M_1} \sqrt{t}}
$$
\end{thm} 

\begin{proof} 
By assumption,
$$
M_1 = \sup_{x\in\mathcal{L}(f,f(x_0))} \max_{i=1,\dots,m} \|\nabla f_i(x)\| > 0.
$$
Following the same proof as above, we obtain
$$
t \min_{s\in[0,t]} \|\mathbf{proj}_{C_\eta(x(s))}(0)\|^2 \leq \frac{1}{M_1}(f_i(x_0) - \inf_{x\in\mathbb{R}^n} f_i(x)),
$$
which establishes the conclusion.
\end{proof}
\section{Accelerated scaled gradient flow}
In this section, we consider the accelerated scaled gradient flow as follows:
\begin{equation}
\ddot x(t)+\frac{r}{t+\theta}\dot x(t)+{\proj}_{C_\alpha(x(t))}(-\ddot x(t))=0
\end{equation}
where  $C_\alpha(x(t)):={\bf conv}\left\{\frac{\nabla f_i(x(t))}{\alpha_i}\mid i=1,\cdots,m\right\}$. 
\begin{lem}
    Suppose that Assumptions 1, 2, and 3 hold. Let $x:[t_0,+\infty )\to \R^n$ be a solution of $\rm (CP)$, and for $i=1,2,\cdots,m$, define the energy function  
    $$
    \mathcal W_i(t)=f_i(x(t))+\frac{\alpha_i}2\|\dot x(t)\|^2.
    $$
    Then $\lim_{t\to \infty }\mathcal W_i(t)$ exists.  
\end{lem}
\begin{proof}
    Differentiating gives  
    $$
    \frac{d}{dt}\mathcal W_i(t)=\Big< \nabla f_i(x(t)),\dot x(t)\Big>+\alpha_i\Big<\dot x(t),\ddot x(t)\Big>.
    $$
    By the projection theorem, we know  
    $$
    \Big<\frac{\nabla f_i(x(t))}{\alpha_i}+\frac{r }{t+\theta}\dot x(t)+\ddot x(t),\dot x(t)\Big>\le 0.
    $$
    Hence, $\frac{d}{dt}\mathcal W_i(t)\le 0$. Since $f_i$ is bounded from below, it follows that $\lim_{t\to \infty }\mathcal W_i(t)$ exists and is finite.
\end{proof}
\begin{coro}
    Let $x:[t_0,+\infty )\to \R^n$ be a solution of $\rm (CP)$ with $\dot x(t_0)=0$. For $i=1,\cdots,m$ and $t\in [t_0,+\infty )$, it holds that  
    $$
    f_i(x(t))\le f_i(x_0),
    $$
    i.e., $x(t)\in \mathcal L(f(x_0)) $ for all $t\ge t_0$. 
\end{coro}
\begin{proof}
    According to the preceding lemma, we have  
    $$
    f_i(x_0)=\mathcal W_i(t_0)\ge \mathcal W_i(t)=f_i(x(t))+\frac{a_i}2\|\dot x(t)\|^2\ge f_i(x(t)).
    $$
\end{proof}
\begin{lem}
    Let $\delta >0$, and let $\omega :[\delta ,+\infty )\to \R$ be a continuously differentiable positive function. Suppose  
    $$
    (t+\theta)\ddot \omega(t)+\alpha \dot \omega (t)\le g(t)
    $$
    for some $\alpha >1$ and almost every $t>\delta $, where $\theta \ge 0$ and $g\in L^1(\delta ,+\infty )$. Then $\lim _{t\to \infty }\omega (t)$ exists.
\end{lem}
\begin{proof}
    Multiplying both sides by $(t+\theta )^{\alpha -1}$, we obtain  
    $$
    \frac{d}{dt}\left[(t+\theta )^\alpha \dot \omega (t)\right]\le (t+\theta)^{\alpha -1}g(t).
    $$
    Integrating yields  
    $$
    [\dot \omega ]_+(t)\le \frac{(\delta+\theta)^\alpha |\dot \omega (\delta )|}{(t+\theta)^\alpha }+\frac{1}{(t+\theta)^\alpha }\int_{\delta }^t(s+\theta)^{\alpha -1}g(s)ds.
    $$
    Thus,  
    $$
    \int_{\delta }^\infty[\dot \omega]_+(t)dt\le \frac{(\delta+\theta)^\alpha |\dot \omega(\delta )|}{(\alpha -1)\delta ^{\alpha -1}}+\int_{\delta }^\infty \frac{1}{(t+\theta)^\alpha}\left(\int_{\delta }^t(s+\theta )^{\alpha -1}g(s)ds\right)dt.
    $$
    By Fubini’s theorem, we deduce  
    \begin{equation} 
    \begin{aligned} 
    \int_{\delta }^\infty \frac{1}{(t+\theta)^\alpha }\left(\int_{\delta }^t(s+\theta )^{\alpha-1}g(s)ds\right)dt&=\int_{\delta }^\infty \left(\int_{s}^\infty \frac{1}{(t+\theta)^\alpha }dt\right)(s+\theta)^{\alpha -1}g(s)ds\\&=\frac{1}{\alpha -1}\int_{\delta }^\infty g(s)ds.
    \end{aligned}
    \end{equation} 
    Hence,  
    $$
    \int_{\delta }^\infty [\dot \omega ]_+(t)dt< +\infty.
    $$
    Since the function  
    $$
    \xi(t)=\omega(t)-\int_{\delta }^t[\dot \omega ]_+(\tau)d\tau
    $$
    is non-increasing and bounded from below, it follows that  
    $$
    \lim_{t\to \infty }\omega (t)=\lim_{t\to \infty }\xi(t)+\int_{\delta}^{+\infty }[\dot \omega ]_+(\tau)d\tau
    $$
    exists.
\end{proof}
\begin{thm}
    Suppose that Assumptions 1, 2, and 3 hold. Let \(x(t)\) be a trajectory solution of \(\rm (CP)\) for \(r \ge 3\). Then  
    $$
    u_0(x(t)) = O\left(1/(t+\theta)^2\right),
    $$
    and in particular, when $r > 3$, we have $t\|\dot x(t)\|^2 \in L^1([t_0, +\infty))$.
\end{thm}
\begin{proof}
    Define the Lyapunov function  
    $$
    \mathcal E_i(t):=\frac{1}{\alpha_i}(t+\theta)^2 (f_i(x(t))-f_i(z))+\frac{1}{2}\|2(x(t)-z)+(t+\theta)\dot x(t)\|^2+\frac{\xi}{2}\|x(t)-z\|^2.
    $$
    Differentiating yields  
    $$
    \begin{aligned}
    \frac{d}{dt}\mathcal E_i(t) &= 2(t+\theta)\frac{(f_i(x(t))-f_i(z))}{\alpha_i}+(t+\theta)^2\left\langle\frac{\nabla f_i(x(t))}{\alpha_i},\dot x(t)\right\rangle \\
    &\quad +\left\langle 2(x(t)-z)+(t+\theta)\dot x(t), 3\dot x(t)+(t+\theta)\ddot x(t)\right\rangle \\
    &\quad +\xi\left\langle \dot x(t), x(t)-z \right\rangle \\
    &= 2(t+\theta)\frac{(f_i(x(t))-f_i(z))}{\alpha_i}+(t+\theta)^2\left\langle\frac{\nabla f_i(x(t))}{\alpha_i},\dot x(t)\right\rangle \\
    &\quad +2\left\langle x(t)-z, (r - (r-3))\dot x(t)+(t+\theta)\ddot x(t)\right\rangle \\
    &\quad +(t+\theta)\left\langle \dot x(t), (r - (r-3))\dot x(t)+(t+\theta)\ddot x(t)\right\rangle \\
    &\quad +\xi\left\langle \dot x(t), x(t)-z \right\rangle \\
    &= 2(t+\theta)\frac{(f_i(x(t))-f_i(z))}{\alpha_i}+(t+\theta)^2\left\langle\frac{\nabla f_i(x(t))}{\alpha_i},\dot x(t)\right\rangle \\
    &\quad +2(t+\theta)\left\langle x(t)-z, \frac{r}{t+\theta}\dot x(t)+\ddot x(t)\right\rangle -2(r-3)\left\langle x(t)-z, \dot x(t)\right\rangle \\
    &\quad +(t+\theta)^2\left\langle \dot x(t), \frac{r}{t+\theta}\dot x(t)+\ddot x(t)\right\rangle - (r-3)(t+\theta)\|\dot x(t)\|^2 \\
    &\quad +\xi\left\langle \dot x(t), x(t)-z \right\rangle.
    \end{aligned}
    $$
    Note that  
    $$
    2(t+\theta)\left\langle x(t)-z, \frac{r}{t+\theta}\dot x(t)+\ddot x(t)\right\rangle \le -2(t+\theta)\min_{i=1,\cdots,m}\frac{(f_i(x(t))-f_i(z))}{\alpha_i},
    $$
    and  
    $$
    \left\langle \frac{\nabla f_i(x(t))}{\alpha_i} + \frac{r}{t+\theta}\dot x(t) + \ddot x(t), \dot x(t) \right\rangle \le 0.
    $$
    Thus,  
    $$
    \frac{d}{dt}\mathcal E(t) + (r-3)(t+\theta)\|\dot x(t)\|^2 \le 0.
    $$
    Finally, we obtain  
    $$
    \sup_{z \in \R^n} \min_{i=1,\cdots,m} \frac{f_i(x(t))-f_i(z)}{\alpha_i} \le \frac{\mathcal E(t_0)}{(t+\theta)^2} \le \frac{\min_{i=1,\cdots,m} \frac{f_i(x(t)) - \inf_{x \in \R^n} f_i}{\alpha_i} + R^2}{(t+\theta)^2}.
    $$
    When \(r > 3\), it follows that  
    $$
    \int_{t_0}^{+\infty} t \|\dot x(t)\|^2 \, dt < +\infty.
    $$
\end{proof}
\begin{thm}
    Suppose that Assumptions 1, 2, and 3 hold. Let \(x(t)\) be a trajectory solution of \(\rm (CP)\). Then \(x(t)\) converges to a weak Pareto optimum of \(\rm (MOP)\). 
\end{thm}
\begin{proof}
    Define the set  
    $$
    S := \{z \in \R^n : f_i(z) \le f_i^\infty\},
    $$
    where \(f_i^\infty = \lim_{t \to \infty} f_i(x(t))\). Since \(x(t)\) is bounded, there exists an accumulation point \(x^\infty \in \R^n\) and a sequence \(\{t_k\}\) such that \(x(t_k) \to x^\infty\) as \(k \to \infty\). By the lower semicontinuity of the objective functions,  
    $$
    f_i(x^\infty) \le \liminf_{k \to \infty} f_i(x(t_k)) = \lim_{k \to \infty} f_i(x(t_k)) = f_i^\infty.
    $$
    Thus, every accumulation point of \(x(t)\) lies in the nonempty set \(S\). Let \(z \in S\) and define \(h_z(t) = \frac{1}{2} \|x(t) - z\|^2\). Then  
    $$
    \ddot h_z(t) + \frac{\alpha}{t+\theta} \dot h_z(t) = \left\langle x(t) - z, \ddot x(t) + \frac{\alpha}{t+\theta} \dot x(t) \right\rangle + \|\dot x(t)\|^2.
    $$
    Note that  
    $$
    f_i(x(t)) + \frac{\alpha_i}{2} \|\dot x(t)\|^2 = \mathcal W_i(t) \ge f_i^\infty.
    $$
    For any \(z \in S\), we have  
    $$
    f_i^\infty \ge f_i(z) \ge f_i(x(t)) + \left\langle \nabla f_i(x(t)), z - x(t) \right\rangle,
    $$
    which implies  
    $$
    \left\langle \frac{\nabla f_i(x(t))}{\alpha_i}, z - x(t) \right\rangle \le \frac{1}{\alpha_i}(f_i(z) - f_i(x(t))) \le \frac{1}{2} \|\dot x(t)\|^2.
    $$
    Therefore,  
    $$
    \left\langle \ddot x(t) + \frac{\alpha}{t+\theta} \dot x(t), x(t) - z \right\rangle \le \frac{1}{2} \|\dot x(t)\|^2.
    $$
    Thus,  
    $$
    (t+\theta) \ddot h_z(t) + \alpha \dot h_z(t) \le \frac{3}{2} t \|\dot x(t)\|^2.
    $$
    Hence, \(\lim_{t \to \infty} \|x(t) - z\|\) converges, which completes the proof of convergence to a weakly Pareto optimal solution.
\end{proof}

\section{Relationship with Discretization Algorithms}

Algorithm  
\begin{equation} 
x_{k+1}=x_k -s_k\textbf{proj}_{C_\alpha (x_k)}(0)\tag{1}
\end{equation} 
where $0<s_{\min}\le s_k\le \min_{i=1,\cdots,m}\frac{2\alpha_i(x_k,k)}{L_i} $. In general, choosing $s_k =\frac{2\alpha_{\min}}{L_{\max}}$ suffices.  

\begin{thm}
Suppose Assumptions 1, 2, and 3 hold. Let $\{x_k\}$ be the sequence of iterates generated by the iterative scheme (1) with initial point $x_0$ and step sizes satisfying $0<s_{\min}\le s_k \le \min _{i=1,\cdots,m}\frac{2\alpha _i(x_k,k)}{L_i}$. Then, $f_i(x_{k+1})\le f_i(x_k)$ for all $i=1,\cdots,m$, and $u_0(x_{k})=O(1/k)$.  
\end{thm}
\begin{proof}
According to the definition of the iterative scheme and using the projection theorem, we obtain  
$$
\left<\frac{\nabla f_i(x_k)}{\alpha_i(x_k,k)}+\frac{(x_{k+1}-x_k)}{s_k},x_{k+1}-x_k\right>\le 0,
$$
so $\left< \frac{\nabla f_i(x_k)}{\alpha_i(x_k,k)},x_{k+1}-x_k\right>\le -\frac{1}{s_k}\|x_{k+1}-x_k\|^2$. Further, we have  
\begin{equation} 
\begin{aligned}
f_i(x_{k+1})-f_i(x_k)&\le \alpha_i(x_k,k)\left(\left<\frac{\nabla f_i(x_k)}{\alpha_i(x_k,k)},x_{k+1}-x_k\right>+\frac{L_i}{2\alpha_i(x_k,k)}\|x_{k+1}-x_k\|^2\right)\\
&\le \frac{\alpha_i(x_k,k)}{2s_k}\left(\frac{L_is_k}{\alpha_i(x_k,k)}-2\right)\|x_{k+1}-x_k\|^2\le 0.
\end{aligned}\tag{2}
\end{equation}
This proves $f_i(x_{k+1})\le f_i(x_k)$. Next, we show that $u_0(x_k )=O(1/k)$.

Compute  
$$
f_i(x_{k+1})-f_i(x_k)\le \left\langle \nabla f_i(x_k), x_{k+1}-x_k \right\rangle + \frac{L_i}{2} \|x_{k+1}-x_k\|^2 \\  
f_i(x_k)-f_i(z)\le \left\langle \nabla f_i(x_k), x_k-z \right\rangle  
$$
Thus, we obtain  
$$
f_i(x_{k+1})-f_i(z)\le \left\langle \nabla f_i(x_k), x_{k+1}-z \right\rangle + \frac{L_i}{2} \|x_{k+1}-x_k\|^2  
$$

Let $\lambda_k := (\lambda_1^k, \cdots, \lambda_m^k)$ such that $x_{k+1}-x_k = -s_k \textbf{proj}_{C(x_k)}(0) = -s_k \sum_{i=1}^m \lambda_i^k \nabla f_i(x_k)$, and take $z \in \mathcal{L}(f, f(x_{k+1}))$, then we can obtain  
\begin{equation} 
\begin{aligned}  
\sum_{i=1}^m \frac{\lambda_i^k}{\alpha_i(x_k,k)}(f_i(x_{k+1})-f_i(z)) &\le \sum_{i=1}^m \frac{\lambda_i^k}{\alpha_i(x_k,k)} \left\langle \nabla f_i(x_k), x_{k+1}-z \right\rangle + \sum_{i=1}^m \frac{\lambda_i^k}{\alpha_i(x_k,k)} \frac{L_i}{2} \|x_{k+1}-x_k\|^2 \\  
&\le \left\langle \sum_{i=1}^m \frac{\lambda_i^k}{\alpha_i(x_k,k)} \nabla f_i(x_k), x_{k+1}-z \right\rangle + \sum_{i=1}^m \frac{\lambda_i^k}{\alpha_i(x_k,k)} \frac{L_i}{2} \|x_{k+1}-x_k\|^2 \\  
&= \frac{1}{s_k} \left\langle x_k - x_{k+1}, x_{k+1}-z \right\rangle + \sum_{i=1}^m \frac{\lambda_i^k}{\alpha_i(x_k,k)} \frac{L_i}{2} \|x_{k+1}-x_k\|^2 \\  
&= -\frac{1}{s_k} \|x_{k+1}-x_k\|^2 + \sum_{i=1}^m \frac{\lambda_i^k}{\alpha_i(x_k,k)} \frac{L_i}{2} \|x_{k+1}-x_k\|^2 + \frac{1}{s_k} \left\langle x_k - x_{k+1}, x_k - z \right\rangle \\  
&= \left( \sum_{i=1}^m \frac{\lambda_i^k}{\alpha_i(x_k,k)} \frac{L_i}{2} - \frac{1}{s_k} \right) \|x_{k+1}-x_k\|^2 + \frac{1}{2s_k} \left[ \|x_k - z\|^2 - \|x_{k+1} - z\|^2 \right] \\  
&\le \frac{1}{2s_k} \left[ \|x_k - z\|^2 - \|x_{k+1} - z\|^2 \right]  
\end{aligned}  
\end{equation} 

Since $f_i(x_{k+1}) \ge f_i(z)$, the right-hand side of the last inequality is greater than or equal to $0$. Based on this, we define $\widetilde{\lambda}_i^k = \frac{\frac{\lambda_i^k}{\alpha_i(x_k,k)}}{\sum_{i=1}^m \frac{\lambda_i^k}{\alpha_i(x_k,k)}}$, and it is easy to obtain  
\begin{equation} 
\begin{aligned}  
\min_{i=1,\cdots,m} (f_i(x_{k+1}) - f_i(z)) &\le \sum_{i=1}^m \widetilde{\lambda}_i^k (f_i(x_{k+1}) - f_i(z)) \\  
&\le \frac{1}{2s_k \cdot \sum_{i=1}^m \frac{\lambda_i^k}{\alpha_i(x_k,k)}} \left[ \|x_k - z\|^2 - \|x_{k+1} - z\|^2 \right] \\  
&\le \frac{\alpha_{\max}}{2s_{\min}} \left[ \|x_k - z\|^2 - \|x_{k+1} - z\|^2 \right]  
\end{aligned}  
\tag{3}  
\end{equation} 

Define an auxiliary function  
$$
E(k) = k \min_{i=1,\cdots,m} \left( f_i(x_k) - f_i(z) \right) + \frac{\alpha_{\max}}{2s_{\min}} \|x_k - z\|^2  
$$
Then  
$$
\begin{aligned}  
E(k+1) - E(k) &= k \left( \min_{i=1,\cdots,m} (f_i(x_{k+1}) - f_i(z)) - \min_{i=1,\cdots,m} (f_i(x_k) - f_i(z)) \right) \\  
&\quad + \min_{i=1,\cdots,m} (f_i(x_{k+1}) - f_i(z)) + \frac{\alpha_{\max}}{2s_{\min}} \|x_{k+1} - z\|^2 - \frac{\alpha_{\max}}{2s_{{\min}}} \|x_k - z\|^2 \\  
&\overset{(*)}{\le} k \max_{i=1,\cdots,m} (f_i(x_{k+1}) - f_i(x_k)) + \frac{\alpha_{\max}}{2s_{{\min}}} \left[ \|x_k - z\|^2 - \|x_{k+1} - z\|^2 \right] \\  
&\quad + \frac{\alpha_{\max}}{2s_{\min}} \left[ \|x_{k+1} - z\|^2 - \|x_k - z\|^2 \right] \\  
&\overset{(**)}{\le} 0  
\end{aligned}  
$$
where $(*)$ holds based on (3), and $(**)$ holds based on the non-increasing property of $\{f_i(x_k)\}$.  

According to the above, for any $p=0,1,\cdots,k$, we have  
$$
E(p+1) - E(p) \le 0  
$$
Summing from $p=0$ to $p=k-1$, we obtain  
$$
k \min_{i=1,\cdots,m} \left( f_i(x_k) - f_i(z) \right) \le E(k) \le E(0) = \frac{\alpha_{\max}}{s_{\min}} \|x_0 - z\|^2  
$$
Then, based on Assumption 3, we get  
$$
u_0(x_k) \le \frac{\alpha_{\max} R^2}{s_{\min} k}
$$
\end{proof}
		%		\input{../append}
	%	\begin{appendices}
	%		\appendix
	%		\section{Auxiliary lemmas}
			
	%	\end{appendices}
%\newpage
		\bibliographystyle{abbrv}

	\end{document}